\documentclass[a4paper,12pt]{article}

\usepackage[utf8]{inputenc}
\usepackage[english]{babel}
\usepackage{hyperref}
\usepackage{amssymb,amsmath,amsthm}
\usepackage{enumerate}
\usepackage{enumitem}
\usepackage{amssymb}
\usepackage{graphicx, color}
\usepackage{epsfig}
\usepackage{tikz}
\usetikzlibrary{calc,decorations.pathmorphing,decorations.text}
\usepackage{subcaption}
\usepackage{refcount}
\usepackage[hmargin=2.5cm,vmargin=3cm]{geometry}
\usepackage{mathtools, braket}
\usepackage[normalem]{ulem}
\usepackage{authblk}
\usepackage{centernot}
\usepackage{indentfirst}
\usepackage{framed} 
\usepackage{amsmath}
\allowdisplaybreaks[4]
\usepackage{ctex}
\usepackage{color}

\usepackage{hyperref}
\hypersetup{
	colorlinks=true,
	linkcolor=blue,
	citecolor=red}

\theoremstyle{plain}
\newtheorem{thm}{Thm}[section]

\newtheorem{theorem}[thm]{Theorem}
\newtheorem{lemma}{Lemma}[section]

\newtheorem{innercustomgeneric}{\customgenericname}
\providecommand{\customgenericname}{}
\newcommand{\newcustomtheorem}[2]{%
	\newenvironment{#1}[1]
	{%
		\renewcommand\customgenericname{#2}%
		\renewcommand\theinnercustomgeneric{##1}%
		\innercustomgeneric
	}
	{\endinnercustomgeneric}
}

\newcustomtheorem{customtheorem}{Theorem}
\newcustomtheorem{customlemma}{Lemma}

\newenvironment{proof*}{\noindent \emph{Proof.}}{\hfill$\Diamond$}

\renewcommand{\pod}[1]{\allowbreak\mathchoice
	{\if@display \mkern 0mu\else \mkern 0mu\fi (#1)}
	{\if@display \mkern 0mu\else \mkern 0mu\fi (#1)}
	{\mkern 1mu(\mathrm{mod}\mkern 4mu #1)}
	{\mkern 0mu(#1)}
}

\usepackage[color=green!30]{todonotes}

\usepackage{caption}
\usetikzlibrary{matrix}
\usetikzlibrary{decorations.markings}
\tikzstyle{vertex}=[circle, draw, fill=black!50,
inner sep=0pt, minimum width=4pt]
\tikzset{->-/.style={decoration={
			markings,
			mark=at position .5 with {\arrow{>}}},postaction={decorate}}}

\tikzstyle{bigblue}=[color=blue, very thick, >=stealth]
\tikzstyle{lightblue}=[color=blue, thin, >=stealth]

\tikzstyle{bigred}=[color=red, very thick, >=stealth]
\tikzstyle{lightred}=[color=red, thin, >=stealth]

\tikzstyle{biggreen}=[color=black!30!green, very thick, >=stealth]
\tikzstyle{lightgreen}=[color=black!30!green,  thin, >=stealth]

\begin{document}
	\title{A new perspective from hypertournaments to tournaments}
 \date{}
\author{Jiangdong Ai\thanks{Corresponding author.  School of Mathematical Sciences and LPMC, Nankai University, Tianjin 300071, P.R.
China. Email: jd@nankai.edu.cn. },~ Qiming Dai\thanks{School of Mathematical Sciences and LPMC, Nankai University, Tianjin 300071, P.R.
China. Email: newbbbie@163.com.},~ Qiwen
Guo\thanks{Center for Combinatorics and LPMC, Nankai University, Tianjin 300071, P.R.
China. Email: gqwmath@163.com.},~ Yingqi Hu\thanks{Center for Combinatorics and LPMC, Nankai University, Tianjin 300071, P.R.
China. Email: nkuhyq@163.com.},~ Changxin Wang\thanks{Center for Combinatorics and LPMC, Nankai University, Tianjin 300071, P.R.
China. Email: Simonang@163.com.}}
\maketitle
\begin{center}
\begin{minipage}{140mm}
\small\noindent{\bf Abstract:}
   A $k$-tournament $H$ on $n$ vertices is a pair $(V, A)$ for $2\leq k\leq n$, where $V(H)$ is a set of vertices, and $A(H)$ is a set of all possible $k$-tuples of vertices, such that for any $k$-subset $S$ of $V$, $A(H)$ contains exactly one of the $k!$ possible permutations of $S$. In this paper, we investigate the relationship between a hyperdigraph and its corresponding normal digraph. Particularly, drawing on a result from Gutin and Yeo, we establish an intrinsic relationship between a strong $k$-tournament and a strong tournament, which enables us to provide an alternative (more straightforward and concise) proof for some previously known results and get some new results. 

\end{minipage}
\end{center}
\section{Introduction} 

A $k$-hyperdigraph $H$ on $n$ vertices is a pair $(V, A)$ for $2\leq k\leq n$, where $V(H)$ is a set of vertices, and $A(H)$ is a set of $k$-tuples of vertices, called hyperarcs, such that for a $k$-subset $S$ of $V$, $A(H)$ contains exactly one of the $k!$ permutations of $S$. For a hyperarc $a=(x_1x_2\dots x_k)$, we say $x_i$ precedes $x_j$ if $i<j$, and write as $x_iax_j$. A 2-hyperdigraph is merely an oriented graph. 
A path $P$ of length $k$ in $H$ is an alternating sequence $P=x_1a_1x_2a_2\dots x_ka_kx_{k+1}$ of distinct vertices $x_i$ and distinct hyperarcs $a_j$ such that $x_ia_ix_{i+1}$ for every $1\leq i\leq k$. We call $P$ a cycle of length $k$ if $x_{k+1}=x_1$. A path (cycle) is Hamiltonian if it contains all vertices of $H$.
A $k$-hyperdigraph $H$ is strong if there is a path from $u$ to $v$ for each ordered pair $(u,v)$, where $u,v$ are distinct vertices in $H$.
A vertex (a hyperarc) of $H$ is pancyclic, if it is contained in an $l$-cycle for all $l\in \{3,\dots, n\}.$ A $k$-hyperdigraph $H$ is vertex (hyperarc)-pancyclic if all of its vertices (hyperarcs) are pancyclic. 

Tournaments are the widely studied class in digraphs. 
We call $H$ a $k$-tournament if $H$ is a $k$-hyperdigraph and each $k$-subset of $V(H)$ has exactly one permutation that belongs to $A(H)$. A tournament is a $2$-hypertournament.

In \cite{Gutin1997}, Gutin and Yeo proved Theorem \ref{Gutin}, an extension of Redei's theorem and Camion's theorem to hypertournaments, which are the most basic results on tournaments. They showed every tournament contains a Hamiltonian path, and every strong tournament has a Hamiltonian cycle. Moreover, they proposed the following interesting question: Is a strong $k$-tournament pancyclic or vertex-pancyclic? 
\begin{theorem}\label{Gutin}\cite{Gutin1997}
For $k\geq 3$, every $k$-tournament on $n\geq k+1$ vertices has a Hamiltonian path, and every strong $k$-tournament on $n\ge k+2$ vertices contains a Hamiltonian cycle.
\end{theorem}

Petrovic and Thomassen \cite{Petrovic2006} proved that a $k$-tournament $H$ on $n$ vertices is vertex-pancyclic if and only if $H$ is strong for sufficiently large $n$, which is an extension of Moon's theorem\cite{Moon1966}. 

Yang \cite{Yang2009} improved their result as follows.
\begin{theorem}\label{Yang}\cite{Yang2009}
Let $H$ be a $k$-tournament on $n$ vertices. When
\begin{enumerate}[label=(\roman*)]
    \item $k = 3$ and $n\geq 15$,
    \item $k = 4$ and $n\geq 11$,
    \item $k\geq 5$ and $n\geq k+4$, or
    \item $k\geq 8$ and $n\geq k+3$,
\end{enumerate}
 $H$ is vertex-pancyclic if and only if $H$ is strong.
\end{theorem}
In 2013, Li, Li, Guo, and Surmacs\cite{Li2013} solved this problem completely. 
\begin{theorem}\label{Li}\cite{Li2013}
When $k\geq 3$ and $n\geq k+2$, an $n$-vertex $k$-tournament $H$ is vertex-pancyclic if and only if $H$ is strong.
\end{theorem}
Besides Hamiltonian property and vertex-pancyclicity, there are some other properties of hypertournaments being studied in a large number of papers, see \cite{Ai2022, Brcanov2013, Guo2014, Li2016, Petrovic2019}.

 Petrovic, Thomassen\cite{Petrovic2006}, and Yang \cite{Yang2009} constructed a certain strong semicomplete digraph $D_H$  from a given strong $k$-tournament $H$ to prove that $H$ is vertex-pancyclic. In other words, they gave another proof of Theorem \ref{Gutin} for some conditions of $k,n$. 
 With the help of Theorem \ref{Gutin} \cite{Gutin1997}, we find a deeper relationship between a strong $k$-tournament and a corresponding strong tournament, which can imply some known results immediately.
\begin{theorem}\label{main theorem hypertournament}
    If $ 3 \le k \le n-3$ and $n\ge 7$, then there is a strong tournament $ T\in \mathcal{T}_H $ where $H$ is a strong $k$-tournament of $n$ vertices. 
\end{theorem}

The outline of the rest of the paper is as follows. 
We give some definitions and lemmas which will be used in the following proofs. 
In Section \ref{sec:$k$-tournament}, we prove Theorem \ref{main theorem hypertournament} and give a more straightforward proof for some previously known properties about strong $k$-tournaments. 
We devote Section \ref{sec:$k$-hyper-digraph} to investigating the connection between $k$-hyperdigraphs and their corresponding digraphs, and extend some results to $k$-hyperdigraphs.

\section{Preliminary}\label{sec:lemma}

The terminology not introduced in this paper can be found in \cite{Bang-Jensen2000}. 
Let $H$ be a $k$-tournament with $n$ vertices where $3\leq k\leq n-2$ and $T$ a tournament with $V(T)=V(H)$. For a hyperarc $a$ of $A(H)$ and an arc $e=(v_i,v_j)$ of $A(T)$, we call $e$ is generated by $a$ if $v_i,v_j\in V(a)$ and $v_i$ precedes $v_j$. Moreover, we say a tournament $T$ is generated by a $k$-tournament $H$ if each arc of $T$ is generated by a unique hyperarc of $H$. We denote by $\mathcal{T}_H$ the set of all tournaments generated by $H$.

For $k\geq 3$ and $n\geq k+2$, let $H$ be a strong $k$-tournament on $n$ vertices. Note that $H$ contains a Hamiltonian cycle $C$ by Theorem \ref{Gutin}, which implies that the hyperarcs of $C$ can generate a Hamiltonian cycle of a strong tournament $T$. If other arcs of $T$ can be generated by the remaining hyperarcs of $H$, then  $T \in \mathcal{T}_H$. To achieve that, we need the following lemmas.

\begin{lemma}\label{AB}
Let $G$ be a bipartite graph with partite sets $U$ and $W$, and $p$ a positive integer. If $d(u)\geq p$ for each $u \in U$ and $d(w)\le p$ for each $w \in W$, then $G$ has a matching covering $U.$
\end{lemma}

\begin{proof}
   Let $S$ be a subset of $U$ and let $E$ be the set of edges of $G$ between $S$ and $N(S)$. Since $d(u)\geq p$ for each $u \in U$ and $d(w)\le p$ for each $w \in W$, we have $$p \left | N(S)\right | \ge  \left | E\right |  \ge  p\left | S\right |.$$ Since $p\ge 1$, it follows that $ \left | N(S)\right |\ge \left | S\right |$. By Hall's theorem \cite{Hall1935}, there is a matching of $G$ covering $U.$
\end{proof}

 \begin{lemma}\label{k=3}
     Let $H$ be a $3$-hypertournament with $n$ vertices and $C$ a Hamiltonian cycle of $H$. We have every pair of nonconsecutive vertices in $C$ that can be contained in at most four hyperarcs of $C$.
     More precisely, when $n=8$ there are at most two nonadjacent pairs contained in four hyperarcs of $C$.    
     When $n=7$, there are no two nonadjacent pairs contained in four hyperarcs of $C$, and at most two nonadjacent pairs contained in at least three hyperarcs of $C$. If there is one nonadjacent pair contained in four hyperarcs of $C$ and one in three hyperarcs, then other pairs are contained in at most one hyperarc.
 \end{lemma}    
 \begin{proof}
      Let $C=v_1 a_1 v_2 \dots a_{n-1} v_n a_n v_1$ and $i,j\in [n]$ where $1< j-i< n-1$. Since $k=3$ and $i,j$ are not consecutive, the hyperarcs of $C$ containing both $v_i$ and $v_j$ must contain exactly one vertex of $\{ v_{i-1}, v_{i+1}, v_{j-1}, v_{j+1}\}$ $\pmod{n}$. So the number of hyperarcs that contain $v_i$ and $v_j$ is at most four. If the number is exactly four, we have $2< j-i< n-2$.

      When $n=8$, assume the nonadjacent pair $(v_i,v_j)$ is contained in four hyperarcs of $C$. By the above discussions, the four hyperarcs are $\{ v_{i-1}, v_i, v_j\}$, $\{ v_i,v_{i+1}, v_j\}$, $\{ v_i,v_{j-1}, v_j\}$, $\{v_i, v_j,v_{j+1}\}$ $\pmod{n}$. 
      If there is the other pair contained in four hyperarcs, it must be $(v_{i+2},v_{j+2})$ $\pmod{n}$ by $2<j-i<n-2$.
      For these two pairs, note that the corresponding hyperarc sets are disjoint. Since when $n=8$, $C$ contains eight hyperarcs, there are at most two nonadjacent pairs contained in four hyperarcs of $C$, and when $n=7$, there are no two nonadjacent pairs contained in four hyperarcs of $C$.

      When $n=7$, $C$ contains seven hyperarcs. For any two nonadjacent pairs contained in at least three hyperarcs of $C$, if their corresponding hyperarc sets are disjoint, then there are at most two nonadjacent pairs contained in at least three hyperarcs of $C$. If not, then there is at most one hyperarc containing them, and two pairs have exactly one common vertex. Assume the pairs are $(v_i,v_j)$ and $(v_i,v_\ell)$, and there is a hyperarc $\{v_i,v_j,v_\ell\}$ in $C$ containing them. Since $v_i$ is not adjacent to $v_j,v_l$ and $\{v_i,v_j,v_\ell\}\in E(C)$, without loss of generality, assume $\ell=j-1$ (module $n$). Some hyperarcs in $C$ are $\{v_i,v_{\ell},v_j\}$, $\{v_i,v_j,v_{j+1}\}$, $\{v_{i-1},v_i,v_j\}$, $\{v_i,v_\ell,v_{i+1}\}$, $\{v_i,v_{\ell-1},v_\ell\}$, where $\ell=j-1$.
      Suppose that there is another nonadjacent pair contained in at least three hyperarcs of $C$, so the pair is contained in some hyperarc which contains $(v_i,v_j)$ or $(v_i,v_\ell)$. Since $(v_i,v_j)$ and $(v_i,v_\ell)$ occur three times, the pair can not contain $v_i$. Then the pair must have one vertex of $\{v_j,v_\ell\}$, and the other vertex is adjacent to $v_i$. Without loss of generality, assume the pair is $(v_{i-1},v_j)$. All possible hyperarcs in $C$ which can contain $(v_{i-1},v_j)$ are $\{v_{i-2},v_{i-1},v_j\}$, $\{v_{i-1},v_{i},v_j\}$, $\{v_{i-1},v_{j-1},v_j\}$, $\{v_{i-1},v_{j},v_{j+1}\}$. Since $\{v_i,v_{j-1},v_j\}$, $\{v_i,v_j,v_{j+1}\}\in E(C)$, there are no three hyperarcs in $C$ containing $(v_{i-1},v_j)$, a contradiction. Hence there are at most two nonadjacent pairs contained in at least three hyperarcs of $C$.
      
      By symmetry, we can assume the nonadjacent pair $(v_1,v_5)$ contained in four hyperarcs of $C$. If there is a nonadjacent pair contained in three hyperarcs of $C$, then the pair has a vertex that is not adjacent to $v_1,v_5$, which is $v_3$. The pair is $(v_3,v_6)$ or $(v_3,v_7)$. By symmetry, assume it is $(v_{3},v_7)$, then the hyperarcs are $\{ v_{7}, v_1, v_5\}$, $\{v_{1}, v_2, v_5\}$, $\{v_{1}, v_4, v_5\}$, $\{v_{1},v_5, v_6\}$, $\{v_{3},v_6, v_7\}$, $\{v_{2},v_3, v_7\}$, $\{v_{3},v_4, v_7\}$, and observe that other pairs occur at most once. Otherwise, all pairs occur at most twice except for $(v_1,v_5)$.
 \end{proof}

\begin{lemma}\label{k=4}
     Let $H$ be a $4$-hypertournament with $7$ vertices and $C$ a Hamiltonian cycle of $H$. We have that any two different nonadjacent pairs can be contained in at most four same hyperarcs of $C$.
\end{lemma}
\begin{proof}
     Let $(v_i,v_j)$ and $(v_k,v_\ell)$ be two nonadjacent pairs. If $ v_i,v_j,v_k,v_\ell$ are four distinct vertices, then there is at most one hyperarc containing all of them. Otherwise $|\{v_i,v_j,v_k,v_\ell\}|=3$, then there are at most four hyperarcs containing $\{v_i,v_j,v_k,v_\ell\}$ as $n=7$.
\end{proof}

We give the following lemma based on Lemma 8 in Yang \cite{Yang2009}.
\begin{lemma}\label{general-matching}
If
\begin{enumerate}[label=(\roman*)]
    \item $k = 3$ and $n \ge 9$,
    \item $k = 4$ and $n \ge 8$,
    \item $k \ge 5$ and $n \ge k+3$,
\end{enumerate}
then \begin{equation*}
    \binom{k}{2} \le \binom{n-2}{k-2} - 4 \quad\text{for $k = 3$}
\end{equation*}
and \begin{equation}\label{k>4}
   \binom{k}{2} \le \binom{n-2}{k-2} - n \quad\text{for $k \ge 4$}
\end{equation}
\end{lemma}

\begin{proof}
    Since $\binom{k}{2}=\binom{3}{2} =3 \le \binom{n-2}{k-2} - 4=n-6$ when $k = 3$ and $n \ge 9$, we only need to check  the case $k \ge 4$. If $n=k+3$, rewrite the inequality (\ref{k>4}) as $$k^2 + k + 6 \le 2\binom{k+1}{k-2}=\frac{1}{3} (k^3 -k).$$ Then $k^3 -3k^2 -4k -18 \ge 0$ holds for $k \ge 5$. For $k=4$, (\ref{k>4}) holds with $n = 8$. Since  increasing $n$ by 1 increases the right-hand side $\binom{n-2}{k-2} - n$ by $\binom{n-1}{k-2} - \binom{n-2}{k-2} - 1 = \binom{n-2}{k-3} -1$ which is non-negative for $k \ge 4, n\ge k+1$, and the left-hand side $\binom{k}{2}$ is independent of $n$, by applying an induction on $n$, the inequality (\ref{k>4}) holds for $k \ge 5, n \ge k+3$ and $k = 4, n \ge 8$.
\end{proof}
\noindent {\bf Remark} when $k=3$ and $n=5, 6$, since $\binom{n-2}{k-2} - 4 \leq 0$, we can construct a strong $3$-hypertournament $H$ such that there is no strong tournament $T\in \mathcal{T}_H$.
\section{Strong $k$-tournament}\label{sec:$k$-tournament}
With the lemmas proved in the previous section, we are ready to prove Theorem \ref{main theorem hypertournament}. Moreover, it gives a more straightforward proof of the vertex-pancyclicity and arc-pancyclicity for strong hypertournaments.

\noindent{\bf Proof of Theorem \ref{main theorem hypertournament}.}
    Let $H$ be a strong $k$-tournament with $n$ vertices where $ 3 \le k \le n-3$ and $n\ge 7$, and let $K_n$ be a complete graph with the same vertex set as $H$. By Theorem \ref{Gutin}, we can assume that $C=v_1 b_1 v_2 \dots b_{n-1} v_n b_n v_1$ is a Hamiltonian cycle in $H$. Consider a bipartite graph $G$ with partite sets $A= E(K_n) \setminus \{v_1v_2,v_2v_3,\dots,v_nv_1\}$ and $B= A(H) \setminus \{b_1,\dots,b_n \}$. For every $a \in A$ and $b \in B$, $G$ has an edge $ab$ if $a \subset\bar{b}$, where $\bar{b}$ denotes the set of vertices of $b$.
    We aim to prove that $G$ has a matching covering $A$. If such a matching exists, then $K_n$ has an orientation $T\in \mathcal{T}_H$ such that $T$ is strong. 

    Notice that the degree of any vertex in $B$ is at most $\binom {k}{2}$. When $k=3$, by Lemma \ref{k=3}, the degree of any vertex in $A$ is at least $\binom{n-2}{k-2}-4$; when $k\ge4$, the degree of a vertex in $A$ is at least $\binom{n-2}{k-2}-n$. By Lemma \ref{AB} and Lemma \ref{general-matching}, we have that $G$ has a matching covering $A$ except for three cases when $k=3, n=7$ or $n=8$, and $k=4, n=7$. Suppose $G$ has no matching covering $A$ for the three remaining cases. By Hall's theorem, there is a subset $S \subseteq A$ such that $\left | N_{G}(S) \right | \le \left | S \right | -1$.
    Let $E$ be the set of edges between $S$ and $N_{G}(S)$.

   \noindent {\bf Case 1}: $k=3$ and $n=7$.
    According to Lemma \ref{k=3}, we consider the following three subcases. Recall that in this case, the degree of any vertex in $B$ is at most $\binom {k}{2}=3$. So, we have that $3|S|-3\geq3|N_G(S)|\geq|E|$.
    
    \noindent {\bf Subcase 1:} There is a vertex in $A$ with degree $\binom{n-2}{k-2}-4$, a vertex in $A$ with degree $\binom{n-2}{k-2}-3$ and all other vertices in $A$ with degree at least $\binom{n-2}{k-2}-1$ in $G$.
    
    We have $$\left |E \right |\geq \left ( \binom{n-2}{k-2}-4 \right ) + \left ( \binom{n-2}{k-2}-3 \right )+ \left ( \binom{n-2}{k-2}-1 \right ) \left ( \left | S\right |-2 \right ) = 4 \left |S \right |-5.$$
    
    When $|S|\geq 3$, we have $|E|>3|S|-2$, a contradiction. 
    When $|S|=1$ or $2$, considering the degree of vertices in $A$, we have $|N(S)|\geq |S|$, a contradiction.
    
    \noindent {\bf Subcase 2:} There is a vertex in $A$ with degree $\binom{n-2}{k-2}-4$ and all other vertices in $A$ with degree at least $\binom{n-2}{k-2}-2$ in $G$.
    
    We have $$\left |E \right |\geq  \left ( \binom{n-2}{k-2}-4 \right ) + \left ( \binom{n-2}{k-2}-2 \right )  \left ( \left | S\right |-1 \right ) = 3 \left |S \right |-2.$$ 
    
    It contradicts to the fact that $3|S|-3\geq|E|$.

    \noindent {\bf Subcase 3:} There are at most two vertices in $A$ with degree $\binom{n-2}{k-2}-3$ and all other vertices in $A$ with degree at least $\binom{n-2}{k-2}-2$ in $G$.\\
    We have $$\left |E \right |\geq 2 \cdot\left ( \binom{n-2}{k-2}-3 \right ) + \left ( \binom{n-2}{k-2}-2 \right )  \left ( \left | S\right |-2 \right ) = 3 \left |S \right |-2.$$
    Similarly, it is a contradiction.

    \noindent {\bf Case 2:} $k=3$ and $n=8$. By Lemma \ref{k=3}, there are at most two vertices in $A$ with degree $\binom{n-2}{k-2}-4$ and all other vertices in $A$ with degree at least $\binom{n-2}{k-2}-3$ in $G$. Therefore,
     $$\left |E \right |\geq  2\cdot \left ( \binom{n-2}{k-2}-4 \right ) + \left ( \binom{n-2}{k-2}-3 \right )  \left ( \left | S\right |-2 \right ) = 3 \left |S \right |-2.$$
     On the other hand, we have $$\left |E \right |\leq \binom {k}{2}\left | N_{G}(S) \right |= 3\left | N_{G}(S) \right |\leq 3 \left |S \right |-3.$$
     Hence, $ 3 \left |S \right |-2\leq \left |E \right |\leq 3 \left |S \right |-3$, a contradiction.

    \noindent {\bf Case 3:} $k=4$ and $n=7$. Since every four-tuple of $\left \{v_1,\dots, v_7\right \}$ will contain a pair of consecutive vertices, the degree of any vertex in $B$ is at most $\binom {k}{2}-1$ in $G$. 
     And by Lemma \ref{k=4}, either there are a vertex in $A$ with degree $\binom{n-2}{k-2}-7$ and all other vertices in $A$ with degree at least $\binom{n-2}{k-2}-4$ in $G$ or there are at most one vertex in $A$ with degree $\binom{n-2}{k-2}-6$ and all other vertices in $A$ with degree at least $\binom{n-2}{k-2}-5$. 
     Therefore, $$\left |E \right |\geq \left ( \binom{n-2}{k-2}-7 \right ) + \left ( \binom{n-2}{k-2}-4 \right )  \left ( \left | S\right |-1 \right ) = 6 \left |S \right |-3.$$ or $$\left |E \right |\geq \left ( \binom{n-2}{k-2}-6 \right ) + \left ( \binom{n-2}{k-2}-5 \right )  \left ( \left | S\right |-1 \right ) = 5 \left |S \right |-1.$$
    On the other hand,
    $$\left |E \right |\leq \left (\binom {k}{2}-1\right )\left | N_{G}(S) \right |= 5\left | N_{G}(S) \right |\leq 5 \left |S \right |-5.$$
  
    Hence, $ \min\{5 \left |S \right |-1,6 \left |S \right |-3\} \leq \left |E \right |\leq 5 \left |S \right |-5$, a contradiction.

 \hfill$\qedsymbol$
     

In the following, we use Theorem \ref{main theorem hypertournament} to give some immediate results which are proved before by different and independent methods\cite{Guo2014,Li2013,Petrovic2006,Yang2009}. 
\begin{theorem}\cite{Guo2014}\label{Guo2014}
    Let $T$ be a strong tournament and $C$ a Hamiltonian cycle in $T$. Then $C$ contains at least three pancyclic arcs. 
\end{theorem}

\begin{theorem}
    If $H$ is a strong $k$-tournament with $n$ vertices for $3\leq k\leq n-3$ and $n\geq 7$, $H$ has following properties:
\begin{enumerate}[label=(\roman*)]
    \item $H$ is vertex-pancyclic.
    \item If $C$ is a Hamiltonian cycle in $H$, then $C$ contains at least three pancyclic hyperarcs.
\end{enumerate}  
\begin{proof}
    By Theorem \ref{main theorem hypertournament}, there exists a strong tournament $T\in \mathcal{T}_H$. We obtain that $T$ is vertex-pancyclic and any Hamiltonian cycle $C$ in $T$ contains at least three pancyclic arcs by the Moon theorem and Theorem \ref{Guo2014}. The corresponding vertices and hyperarcs in $H$ have the same property, then $H$ is vertex-pancyclic and the corresponding Hamiltonian cycle in $H$ contains at least three pancyclic hyperarcs.
\end{proof}
\end{theorem}

\section{Concluding remarks}\label{sec:$k$-hyper-digraph}

\begin{enumerate}
    \item [1.]

Let's explore the relationship between hyperdigraphs and their corresponding digraphs. This perspective on hyperdigraphs allows us to address certain problems effectively. As an illustration, we will establish an extension of the Gallai-Milgram theorem to $k$-hyperdigraphs.

The \emph{path covering number} of a $k$-hyperdigraph $H$ denoted by $pc(H)$, is the minimum positive integer $m$ such that there are $m$ disjoint paths covering the vertex set of $H$. 
An independent set $I$ of $H$ is a set of vertices such that the induced sub-hyperdigraph of $I$ has no hyperarcs. The \emph{independence number} of $H$, $\alpha(H)$, is the maximum integer $m$ such that $H$ has an independent set of size $m$.

\begin{theorem}[Gallai-Milgram theorem]\label{lpc}\cite{Gallai1960}
    For every digraph $D$, the path covering number is at most its independence number, that is $pc(D) \le \alpha(D)$.
\end{theorem}
\begin{theorem}
   For every $k$-hyperdigraph $H$, $pc(H)\leq \alpha(H)$.
\end{theorem}
\begin{proof}
    Construct a digraph $D$ with $V(D)=V(H)$. Let each hyperarc of $H$ generate an arc of $D$, and delete the parallel arcs. We call such a digraph $D$ generated by $H$. For any path $P$ in $D$, there is a path $P'$ in $H$ such that every arc of $P$ is generated by hyperarcs of $P'$. Hence, we have $pc(H) \le pc(D)$. By Theorem \ref{lpc}, we know $pc(D)\leq \alpha(D)$. On the other hand, if $I$ is an independent set of $D$, then it is also an independent set of $H$. Otherwise, there is a hyperarc $e$ in $H[I]$. By the definition of $D$, there must be an arc $e'$ in $D[I]$ generated by $e$ which contradicts that $I$ is an independent set of $D$. Thus, $$pc(H) \le pc(D)\leq\alpha(D) \le \alpha (H).$$
\end{proof}
As we can see above, it is not hard to extend a property of digraphs to hyperdigraphs, it is natural to ask what else we can do.
\item [2.]
Based on the main result of this paper we could always find a degenerated strong tournament from a strong $k$-tournament, we could extend many properties of strong tournaments to $k$-tournaments besides what we give in Section~\ref{sec:$k$-tournament}. For example, it is well-known that there are at least three $2$-kings in a strong tournament, and so is a strong $k$-tournament. 
\end{enumerate}
\section{Acknowledgement}
This work was partially supported by the National Natural Science Foundation of China (No. 12161141006), the Natural Science Foundation of Tianjin (No. 20JCJQJC00090), and the Fundamental Research Funds for the Central Universities, Nankai University (No. 63231193).

\end{document}